\documentclass[12pt]{article}

\usepackage[cp1251]{inputenc}
\usepackage[english]{babel}
\usepackage{ucs}

\usepackage{amssymb, amsthm, amsmath, latexsym, txfonts, lmodern, mathrsfs, wasysym}

\usepackage{graphicx, wrapfig, caption, subcaption, tikz}

\usepackage{algorithm, algorithmicx}

\usepackage{parskip}
\usepackage{multicol, enumerate, paralist, titlesec, mathtools}
\usepackage[numbers,sort&compress]{natbib}

\usepackage[hidelinks]{hyperref}

\usetikzlibrary{decorations.pathreplacing}

\newtheorem{theorem}{Theorem}
\numberwithin{theorem}{section}

\newtheorem{conjecture}{Conjecture}
\numberwithin{conjecture}{section}

\numberwithin{corollary}{section}

\numberwithin{example}{section}

\newtheorem{lemma}{Lemma}
\numberwithin{lemma}{section}

\numberwithin{proposition}{section}

\numberwithin{problem}{section}

\newtheorem{remark}{Remark}
\numberwithin{remark}{section}

\newtheorem{claim}{Claim}
\numberwithin{claim}{theorem}

\newtheorem{definition}{Definition}
\numberwithin{definition}{section}

\def\qed{\hfill\ifhmode\unskip\nobreak\fi\qquad\ifmmode\Box\else\hfill$\Box$\fi}

\author{
Hadeel Al Bazzal\thanks{KALMA, Faculty of Sciences, Lebanese University, Baalbek, Lebanon; LIB, Universit\'e Bourgogne Europe, Dijon, France}
\and
Olivier Togni\thanks{LIB, Universit\'e Bourgogne Europe, Dijon, France}
}

\usepackage{fullpage}
\begin{document}

\title{t-tone colorings of outerplanar and Halin graphs}
\maketitle

\begin{abstract}
A $t$-tone $k$-coloring of a graph $G$ assigns a set of $t$ distinct colors from $\{1, \dots, k\}$ to each vertex so that vertices at distance $d$ share fewer than $d$ common colors. The $t$-tone chromatic number of $G$ is the minimum $k$ such that $G$ has a $t$-tone $k$-coloring. This paper investigates the $t$-tone coloring of two specific subclasses of planar graphs: subcubic outerplanar graphs and Halin graphs. We provide a complete characterization of the $2$-tone chromatic number for subcubic outerplanar graphs and establish a sharp upper bound for their $3$-tone chromatic number. We then turn to Halin graphs and prove that every cubic Halin graph of order $n \ge 6$ is $2$-tone $7$-colorable. Moreover, we derive an upper bound on the $2$-tone chromatic number for Halin graphs with arbitrary maximum degree.
\end{abstract}



\section{Introduction}
All graphs considered in this paper are assumed to be simple and finite. For a graph $G$, we let $V(G)$ denote its vertex set. The minimum degree and maximum degree of $G$ are denoted by $\delta(G)$ and $\Delta(G)$, respectively. For any distinct vertices $u, v \in V(G)$, the distance between $u$ and $v$ in $G$, denoted by $d_G(u,v)$, is the length of a shortest $uv$-path in $G$. The degree of a vertex $v$ is denoted by $d_G(v)$. The neighbors of $v$ are the vertices at distance one from $v$, and the second neighbors of $v$ are the vertices at distance two from $v$, forming the set $N_G^2(v)$. We write $C_n$, $P_n$, and $K_n$ for the cycle, path, and complete graph on $n$ vertices, respectively. The wheel graph $W_d$ is obtained by joining a single vertex to every vertex of a cycle $C_{d-1}$. We denote by $\binom{\{1,\dots,k\}}{t}$ the family of all $t$-element subsets of $\{1,\dots,k\}$.

\noindent A graph $G$ is subcubic if $\Delta(G) \le 3$. An outerplanar graph is a graph that admits a planar embedding in which every vertex lies on the boundary of the outer face. In such an embedding, all bounded regions are called bounded faces. Furthermore, a Halin graph $H = T \cup C$ is a plane graph consisting of a tree $T$ with no vertices of degree two, and a cycle $C$ that connects the leaves of $T$ according to the cyclic order imposed by the planar embedding of $T$.

\noindent The primary focus of this paper is the $t$-tone $k$-coloring, a generalization of proper coloring. 
\begin{definition}
Let $G$ be a graph, and let $t$ and $k$ be positive integers with $t\le k$. A \emph{$t$-tone $k$-coloring} of $G$ is a function $f \colon V(G) \to \binom{\{1, \dots, k\}}{t}$ such that $|f(u) \cap f(v)| < d_G(u,v)$ for all distinct vertices $u,v \in V(G)$.  A graph $G$ that admits a $t$-tone $k$-coloring is said to be \emph{$t$-tone $k$-colorable}. The \emph{$t$-tone chromatic number} of $G$, denoted $\tau_t(G)$, is the smallest integer $k$ for which $G$ is $t$-tone $k$-colorable.
\end{definition}
\noindent For a given $t$-tone coloring $f$, we call $f(v)$ the label of $v$, and the elements within $\{1, \dots, k\}$ colors. To simplify the notation, we frequently omit set brackets and commas; for instance, we write $ab$ instead of $\{a, b\}$ and $abc$ instead of $\{a, b, c\}$. Observe that for any fixed $t$, the parameter $\tau_t$ is monotone: if $H \subseteq G$, then $\tau_t(H) \le \tau_t(G)$.

\noindent The notion of $t$-tone $k$-coloring was introduced by Chartrand and first studied in a research group directed by Zhang, consisting of Fonger, Goss, Phillips, and Segroves \cite{F}. It was then thoroughly explored by Bickle and Phillips \cite{BP}.

\noindent The most widely studied case of $t$-tone coloring is when $t=2$. In this setting, each vertex is assigned a $2$-element label such that adjacent vertices receive disjoint labels, and vertices at distance two receive distinct labels. Several basic results follow immediately; for instance, $\tau_2(K_n)=2n$ and $\tau_2(K_4-e)=7$.  

\noindent For cubic graphs, Bickle and Phillips \cite{BP} proposed the following conjecture.

\begin{conjecture}\label{conj}
Let $G$ be a cubic graph. Then:
\begin{enumerate}
    \item $\tau_2(G)\le 8$.
    \item If $G$ does not contain $K_4$, then $\tau_2(G)\le 7$.
    \item If $G$ does not contain $K_4-e$, then $\tau_2(G)\le 6$.
\end{enumerate}
\end{conjecture}

\noindent In \cite{C1}, Cranston, Kim, and Kinnersley confirmed the first part of the conjecture and disproved the third by presenting a counterexample, namely the Heawood graph, which does not contain $K_4-e$. Since then, $2$-tone coloring has been investigated for several graph classes, including cycles \cite{BP}, trees \cite{F, BP}, wheels \cite{B1}, and others \cite{r, BN, B2, B3, B4, BP, C1, C2, Dong1, Dong2, WuThesis, YangThesis}. Moreover, $2$-tone colorings of graph products were studied in \cite{B1, Cooper, Loe}. General upper bounds for the $2$-tone chromatic number were established in \cite{B1, C1, C2, Dong1, Dong2}, while lower bounds were considered in \cite{Pan}. Finally, $t$-tone coloring for general values of $t$ has been explored in \cite{C2, WuThesis, YangThesis}.

\noindent In this paper, we investigate the $t$-tone coloring of subcubic outerplanar graphs and Halin graphs. The paper is organized as follows. In Section $2$, we examine subcubic outerplanar graphs. We show that a subcubic outerplanar graph $G$ admits a $2$-tone $5$-coloring if and only if $G$ contains no cycles of length $3$, $4$, or $7$. Then we prove that a subcubic outerplanar graph $G$ admits a \emph{good} $2$-tone $6$-coloring if and only if $G$ does not contain $K_4-e$, where a good $2$-tone $k$-coloring is a $2$-tone $k$-coloring in which any two vertices at distance two share exactly one color. These results together yield a complete characterization of the $2$-tone chromatic number of connected subcubic outerplanar graphs of order $n \ge 3$.
\begin{theorem}
Let $G$ be a connected subcubic outerplanar graph of order $n \ge 3$. Then
\[
\tau_2(G)=
\begin{cases}
5, & \text{if } G \text{ contains no } C_3,C_4,C_7,\\
6, & \text{if } G \text{ contains one of } C_3,C_4,C_7 \text{ but no } K_4-e,\\
7, & \text{otherwise}.
\end{cases}
\]
\end{theorem}
\noindent Next, we establish a sharp upper bound for the $3$-tone chromatic number of subcubic outerplanar graphs by proving that $\tau_3(G)\le 11$ for every such graph $G$. Rather than working directly with $3$-tone colorings, we introduce a strengthened coloring framework, called a \emph{$3$-good $k$-coloring}, which assigns a set of three colors from $\{1,2,\dots,k\}$ to each vertex of $G$ such that any two adjacent vertices share no colors, and any two vertices at distance two share exactly one color. This exact intersection condition ensures that all $3$-tone constraints are automatically satisfied, while significantly simplifying the analysis of color interactions at distance three. We show that every subcubic outerplanar graph admits a $3$-good $11$-coloring, which immediately implies the desired upper bound on $\tau_3(G)$. We conclude this section by proposing a conjecture that relates $\tau_t(G)$ to $\tau_{t-1}(G)$ for subcubic outerplanar graphs.

\noindent In Section $3$, we turn our attention to Halin graphs. We show that every cubic Halin graph of order $n \ge 6$ admits a $2$-tone $7$-coloring and then derive a general upper bound on the $2$-tone chromatic number of Halin graphs with arbitrary maximum degree. While the latter result applies to all Halin graphs, the bound obtained in the cubic case is strictly stronger. From an asymptotic perspective, our bound for Halin graphs scales as $\sqrt{2\Delta}$, improving on known bounds for general planar graphs.

\section{$2$-tone and $3$-tone colorings of subcubic outerplanar graphs}

This section examines subcubic outerplanar graphs, presenting a complete characterization of their $2$-tone chromatic number and establishing a sharp upper bound for their $3$-tone chromatic number.

\noindent Cranston and LaFayette~\cite{C2} proved that for any outerplanar graph $G$, \[
\tau_2(G) \le 
\left\lceil 
\sqrt{2\Delta(G) + 4.25}
+ 5.5
\right\rceil.
 \] Later, Bickle~\cite{B2} improved this bound by showing that if $G$ is an outerplanar graph of order $n \geq 4$, then $\tau_2(G) \le \max\left\{7,\;\left\lceil \frac{5 +\sqrt{1 + 8(\Delta(G) + 3)}}{2} \right\rceil\right\}$, which in particular implies that $\tau_2(G) \le 7$ for $\Delta(G)\le3$.

\noindent Given an outerplane graph $G$, the \emph{weak dual} $\tau(G)$ is the graph whose vertices correspond to the bounded faces of $G$, with two vertices adjacent whenever the corresponding faces share an edge. For a bounded face $F$ of $G$, we denote by $C(F)$ the cycle of $G$ that bounds $F$. It is well known that a plane graph is outerplane if and only if its weak dual is a forest. Let $G$ be a subcubic outerplane graph. A bounded face $F$ of $G$ is called a \emph{pendant} face if $F$ corresponds to a leaf of $\tau(G)$ and $C(F)$ contains either exactly one vertex of degree three or exactly two consecutive vertices of degree three. 

\noindent Since outerplane graphs may contain cycles, we recall the following lemma:

\begin{lemma}[\cite{BP}]\label{l1}
For the cycle $C_n$,
\[
\tau_2(C_n) =
\begin{cases}
6, & \text{if } n \in \{3, 4, 7\},\\
5, & \text{otherwise.}
\end{cases}
\]
\end{lemma}

\begin{theorem}\label{t1}
   A subcubic outerplanar graph $G$ has a $2$-tone $5$-coloring if and only if $G$ does not contain $C_3$, $C_4$, or $C_7$.
\end{theorem}

\begin{proof} 
Suppose that $G$ is a subcubic outerplane graph admitting a 2-tone 5-coloring; that is, $\tau_2(G)\le 5$. By the monotonicity of the $2$-tone chromatic number with respect to subgraph inclusion and by Lemma~\ref{l1}, any graph containing $C_3$, $C_4$, or $C_7$ has $2$-tone chromatic number at least $6$. Consequently, $G$ contains none of these cycles.

\noindent Conversely, suppose, for the sake of contradiction, that the statement is false. Then there exists a subcubic outerplane graph that contains no cycles of length $3$, $4$, or $7$ and does not admit a $2$-tone $5$-coloring. Among all such graphs, let $G$ be one with the minimum number of vertices. 

\noindent Clearly, $G$ must be connected; otherwise, each connected component of $G$ would admit a $2$-tone $5$-coloring, and these colorings could be combined to obtain a $2$-tone $5$-coloring of $G$, contradicting the choice of $G$.

\begin{claim}
    $\delta(G)\geq 2$.
\end{claim} 

\begin{proof}
Suppose, to the contrary, that $G$ contains a vertex $u$ of degree one, and let $v$ be its unique neighbor in $G$. Let $G' = G - \{u\}$. Clearly, $G'$ is a subcubic outerplane graph that does not contain any of $C_3$, $C_4$, or $C_7$. By the minimality of $G$, the graph $G'$ admits a $2$-tone $5$-coloring. Since $u$ is adjacent only to $v$, there remain three other colors, giving three possible labels for $u$. Since $u$ has at most two second neighbors in $G$, at least one of these possible labels can be assigned to $u$, a contradiction.\end{proof}

\noindent Note that $G$ must contain at least two bounded faces; otherwise, $G$ would be a cycle, which admits a $2$-tone $5$-coloring by Lemma~\ref{l1}, a contradiction.

\noindent Let $F$ be a pendant face in $G$ with $C(F)=v_1v_2\ldots v_\ell v_1$, where $\ell\ge 5$ and $\ell\neq 7$. Let $L\subseteq V(C(F))$ denote the set of vertices of degree two in $G$ belonging to this face, and define $G' = G-L$. Clearly, $G'$ is a subcubic outerplane graph containing none of $C_3$, $C_4$, or $C_7$. By the minimality of $G$, the graph $G'$ admits a $2$-tone $5$-coloring $f$. By Lemma~\ref{l1}, the cycle $C(F)$ admits a $2$-tone $5$-coloring. By permuting the colors of the $2$-tone $5$-coloring of $C(F)$, or by applying an automorphism of $C(F)$, there exists a $2$-tone $5$-coloring $g$ of $C(F)$ such that $f\cup g$ is a valid $2$-tone $5$-coloring of $G$, a contradiction.\end{proof}

\begin{definition}
    A \emph{good} $2$-tone $k$-coloring of a graph $G$ is a $2$-tone $k$-coloring of $G$ in which any two vertices at distance two share exactly one common color. 
\end{definition}

\begin{lemma}\label{l2}
Every cycle admits a good $2$-tone $6$-coloring.
\end{lemma}

\begin{proof}
In any $2$-tone $5$-coloring, vertices at distance two must share a color; otherwise, disjoint labels would leave no feasible label for their common neighbor. Consequently, every $2$-tone $5$-coloring of a graph $G$ is a good $2$-tone $5$-coloring. Thus, by Lemma~\ref{l1}, each $C_n$ with $n\notin\{3,4,7\}$ admits a good $2$-tone $5$-coloring. For $n\in\{3,4,7\}$, the following are explicit good $2$-tone $6$-colorings:
\[
\begin{array}{rl}
C_3: & -12-34-56-,\\
C_4: & -12-34-15-36-,\\
C_7: & -12-34-15-24-13-25-36-.
\end{array}\]\end{proof}

\begin{theorem}\label{t2}
A subcubic outerplanar graph $G$ admits a good $2$-tone $6$-coloring if and only if $G$ does not contain $K_4-e$.
\end{theorem}

\begin{proof} 
Suppose that $G$ is a subcubic outerplane graph admitting a good $2$-tone $6$-coloring. Since any graph containing $K_4-e$ as a subgraph has $2$-tone chromatic number at least $7$, $G$ does not contain $K_4-e$.

\noindent Conversely, suppose, to the contrary, that the statement is false. Let $G$ be a minimal counterexample. Clearly, $G$ is connected.

\begin{claim}
    $\delta(G)\geq 2$.
\end{claim} 
\begin{proof}
\noindent Suppose, to the contrary, that $G$ contains a vertex $u$ of degree one, and let $G' = G-\{u\}$. Clearly, $G'$ is a subcubic outerplane graph that does not contain $K_4-e$. By the minimality of $G$, $G'$ admits a good $2$-tone $6$-coloring $f$. Let $v$ be the unique neighbor of $u$ in $G'$. Consider the following two cases:

\noindent Case 1. $d_{G'}(v) = 1$.\\
\noindent Let $v'$ be the unique neighbor of $v$ in $G'$. Assign $u$ a label consisting of one color from $f(v')$ and one color from $\{1,\dots,6\}\setminus (f(v)\cup f(v'))$. This extends $f$ to a good $2$-tone $6$-coloring of $G$, a contradiction.

\noindent Case 2. $d_{G'}(v) = 2$.\\
\noindent Let $v'$ and $v''$ be the two neighbors of $v$ in $G'$. Since the distance between $v'$ and $v''$ is at most two, their labels are distinct. Thus, assign $u$ a label consisting of one color from $f(v')\setminus f(v'')$ and a (different) color from $f(v'')\setminus f(v')$. This extends $f$ to a good $2$-tone $6$-coloring of $G$, a contradiction.\end{proof}

\noindent Note that $G$ must contain at least two bounded faces; otherwise, $G$ would be a cycle, which admits a good $2$-tone $6$-coloring by Lemma~\ref{l2}, a contradiction. Let $F$ be a pendant face in $G$ with $C(F) = v_1v_2 \dots v_\ell v_1$, where $\ell \ge 3$.

\begin{claim}
   $\ell \leq 4$.
\end{claim}

\begin{proof}
Suppose, for a contradiction, that $\ell \ge 5$. By the definition of a pendant face, we may assume without loss of generality that either both $v_1$ and $v_\ell$ have degree three, or only $v_1$ has degree three. Consider the graph $G' = (G - \{v_2\}) + v_1v_3$. Since $\ell > 4$, this operation does not create a copy of $K_4 - e$ in $G'$. By the minimality of $G$, the graph $G'$ admits a good $2$-tone $6$-coloring $f$. Let $y$ be the neighbor of $v_1$ that does not belong to $C(F)$. Without loss of generality, assume that $f(v_1) = 12$ and $f(v_3) = 34$, and set $S = \{5,6\}$. Since $f$ is a good $2$-tone $6$-coloring of $G'$, $v_4$ shares exactly one color with $v_1$. Assume, without loss of generality, that this color is $1$. Moreover, each of $v_\ell$ and $y$ shares exactly one color with $v_3$. Hence, the second color of each of $v_4$, $v_\ell$, and $y$ must belong to $S$. We therefore assign $f(v_2) = 56$.

\noindent It remains to ensure that $v_1$ and $v_3$ share exactly one color in $G$. Observe that $d_G(v_4)=2$. Hence, $v_3$ has exactly two second neighbors in $G$, namely $v_1$ and $v_5$. Since $f$ is a good $2$-tone $6$-coloring, $v_3$ and $v_5$ share exactly one color; without loss of generality, assume it is $3$. If $2 \in f(v_5)$, we relabel $f(v_3)$ as $24$; otherwise, we relabel it as $23$. Thus, $f$ extends to a good $2$-tone $6$-coloring of $G$, a contradiction.\end{proof}

\begin{claim}
   $\ell= 3$
\end{claim}

\begin{proof}
 \noindent Suppose, to the contrary, that $\ell=4$. If $d_G(v_1)=3$ and $d_G(v_i)=2$ with $i\in \{2,3,4\}$, then let $G' = G - \{v_2,v_3,v_4\}$. By the minimality of $G$, $G'$ admits a good $2$-tone $6$-coloring $f$. Let $y$ denote the neighbor of $v_1$ in $G'$, and suppose, without loss of generality, that $f(v_1)=12$ and $f(y)=34$. Assigning the labels $35$, $16$, and $45$ to $v_2$, $v_3$, and $v_4$, respectively, yields a good $2$-tone $6$-coloring of $G$, a contradiction.
 
 \noindent Hence, we may assume, without loss of generality, that $d_G(v_1)=d_G(v_2)=3$ and $d_G(v_3)=d_G(v_4)=2$. Let $G' = G - \{v_3,v_4\}$. By the minimality of $G$, $G'$ admits a good $2$-tone $6$-coloring $f$. Let $y_1$ and $y_2$ denote the neighbors of $v_1$ and $v_2$ not in $C(F)$, respectively. Suppose, without loss of generality, that $f(v_1)=12$ and $f(v_2)=34$. If $y_1 = y_2$, then $f(y_1)=56$. Assigning the labels $16$ and $35$ to $v_3$ and $v_4$, respectively, produces a good $2$-tone $6$-coloring of $G$, a contradiction. Hence, $y_1$ and $y_2$ must be distinct. Since $G'$ admits a good $2$-tone $6$-coloring, each of $y_1$ and $y_2$ shares exactly one color with $v_2$ and $v_1$, respectively, and takes its remaining color from the set $S$. Without loss of generality, assume that $f(y_1)\cap f(v_2) = \{3\}$, $f(y_2)\cap f(v_1) = \{1\}$, and $f(y_1)\cap S = \{5\}$, that is, we may assume $f(y_1)=35$. Now assign $45$ to $v_4$ and assign $16$ to $v_3$ if $6\notin f(y_2)$, and $26$ otherwise. This yields a good $2$-tone $6$-coloring of $G$, again a contradiction.\end{proof}

\noindent Then $F$ is a triangle. First, suppose, without loss of generality, that $d_G(v_1)=3$ and $d_G(v_2)=d_G(v_3)=2$. Let $G' = G - \{v_2, v_3\}$. By the minimality of $G$, $G'$ admits a good $2$-tone $6$-coloring $f$. Let $y$ denote the neighbor of $v_1$ in $G'$. Suppose, without loss of generality, that $f(v_1)=12$ and $f(y)=34$. Assigning the labels $35$ and $46$ to $v_2$ and $v_3$, respectively, yields a good $2$-tone $6$-coloring of $G$, a contradiction.

\noindent Therefore, we may assume, without loss of generality, that $d_G(v_1)=d_G(v_2)=3$ and $d_G(v_3)=2$. Let $G' = G - \{v_3\}$. By the minimality of $G$, the graph $G'$ admits a good $2$-tone $6$-coloring $f$. Suppose, without loss of generality, that $f(v_1)=12$ and $f(v_2)=34$. Let $y_1$ and $y_2$ be the neighbors of $v_1$ and $v_2$, not in $C(F)$, respectively. As $G$ does not contain $K_4-e$, $y_1$ and $y_2$ are distinct. Let $S=\{5,6\}$. Because $G'$ admits a good $2$-tone $6$-coloring, $f(y_2)$ and $f(y_1)$ share one color with $f(v_1)$ and $f(v_2)$, respectively, taking their remaining color from $S$. Hence, we may assign $56$ to $v_3$ and obtain a good $2$-tone $6$-coloring of $G$, a contradiction.\end{proof}

\noindent Theorem \ref{t2} establishes that part (3) of Conjecture~\ref{conj} is valid for the class of subcubic outerplanar graphs. Collectively, Theorem \ref{t1} and Theorem \ref{t2} provide the following complete characterization of the $2$-tone chromatic number of connected subcubic outerplanar graphs of order $n \ge 3$.

\begin{theorem}
Let $G$ be a connected subcubic outerplanar graph of order $n \ge 3$. Then
\[
\tau_2(G) =
\begin{cases}
5, & \text{if } G \text{ contains none of } C_3, C_4, C_7,\\
6, & \text{if } G \text{ contains one of } C_3, C_4, C_7 \text{ but not } K_4-e,\\
7, & \text{if } G \text{ contains } K_4-e.
\end{cases}
\]
\end{theorem}

\begin{proof}
Since $G$ is connected with $n \ge 3$, it contains $P_3$ as a subgraph. As $\tau_2(P_3)=5$, it follows that $\tau_2(G)\ge 5$. 

\noindent If $G$ contains no cycle from $\{C_3, C_4, C_7\}$, Theorem~\ref{t1} guarantees a $2$-tone $5$-coloring, and hence $\tau_2(G) = 5$. Suppose $G$ contains such a cycle but does not contain $K_4 - e$. Then Theorem~\ref{t2} gives $\tau_2(G) \le 6$, while the presence of the cycle ensures $\tau_2(G) \ge 6$ (by Lemma~\ref{l1}), so $\tau_2(G) = 6$. Finally, if $K_4-e \subseteq G$, we have $\tau_2(G) \ge 7$ and $n \ge 4$, so the upper bound of Bickle~\cite{B2}, specialized to $\Delta(G) \le 3$, gives $\tau_2(G) \le 7$. Therefore, $\tau_2(G) = 7$, completing the proof.
\end{proof}

\begin{definition}
 A \emph{$3$-good $k$-coloring} of a graph $G$ is an assignment of a set of three colors from $\{1,2,\dots,k\}$ to each vertex of $G$ such that any two adjacent vertices share no colors, and any two vertices at distance two share exactly one color. 
\end{definition}

\begin{remark}
Any $3$-good $k$-coloring of a graph $G$ is a $3$-tone $k$-coloring of $G$. Indeed, since any two vertices at distance two share exactly one color, no two vertices at distance three share more than two colors.
\end{remark}

\begin{lemma}\label{l3}
Every cycle admits a $3$-good $11$-coloring.
\end{lemma}

\begin{proof}
We proceed by induction on the length $n$ of the cycle. For $n=3$ and $n=4$, define a $3$-good $11$-coloring $f$ as follows. If $n=3$, set $f(v_1)=123$, $f(v_2)=456$, and $f(v_3)=789$. If $n=4$, set $f(v_1)=123$, $f(v_2)=456$, $f(v_3)=178$, and $f(v_4)=4910$.
\noindent Assume now that every cycle of length less than $n$ admits a $3$-good $11$-coloring. Let $C = v_1v_2\ldots v_nv_1$ be a cycle of length $n>4$, and let $C' = (C - \{v_2\}) + v_1v_3$. By the induction hypothesis, $C'$ admits a $3$-good $11$-coloring $f$. Without loss of generality, we may assume that $f(v_1)=123$ and $f(v_3)=456$. Let $S=\{7,8,9,10,11\}$. Since $f$ is a $3$-good $11$-coloring of $C'$, $v_4$ shares exactly one color with $v_1$ and takes its two remaining colors from $S$. Hence, without loss of generality, we may set $f(v_4) = 178$. Similarly, $v_n$ shares exactly one color with $v_3$, say $4$, and takes its two remaining colors from $S$. Therefore, the four sets $f(v_1), f(v_3), f(v_4), f(v_n)$ together involve at most ten colors, leaving one unused in $\{1,\ldots, 11\}$, which we may assume, without loss of generality, to be $11$. Observe that $v_n$ may share two, one, or no colors with $v_4$. Without loss of generality, we may assume that $f(v_n) = 478$, $479$, or $4910$, respectively. In all cases, we assign $f(v_2) = 8911$. Note that $v_3$ has to be relabeled to share one color with $v_1$ in $C$. Let $x$ denote the other second neighbor of $v_3$ in $C$ distinct from $v_1$. Since $f$ is a $3$-good $11$-coloring, $x$ and $v_3$ share a common color; without loss of generality, let it be $4$. If $2\in f(x)$, we relabel $v_3$ by $256$; otherwise, by $425$. Then $C$ has a $3$-good $11$-coloring.\end{proof}

\begin{theorem}
    Every subcubic outerplanar graph $G$ has a $3$-good $11$-coloring.
\end{theorem}
 
\begin{proof} 
Suppose, for the sake of contradiction, that the statement is false, and let $G$ be a counterexample of minimum order. Clearly, $G$ is connected.
 
\begin{claim}
    $\delta(G) \ge 2$.
\end{claim}

\begin{proof}
Assume for contradiction that $G$ contains a vertex $u$ of degree one, and let $v$ be its single neighbor. Consider the graph $G' = G - \{u\}$. By the minimality of $G$, the smaller graph $G'$ admits a $3$-good $11$-coloring, denoted by $f$.\\
\noindent If $v$ has only one neighbor $v'$ in $G'$, then choose one color from the label of $v'$ and select two from the remaining colors from $\{1,\dots,11\}$ not used by $f(v)$ and $f(v')$, and assign these three colors to $u$. This yields a valid extension of $f$, a contradiction.\\
\noindent Suppose instead that $v$ has two distinct neighbors, $v'$ and $v''$. Since the distance between $v'$ and $v''$ is at most two, their two labels are distinct. We may then assign to $u$ three colors: $a \in f(v')$, $b \in f(v'')$ with $a \ne b$, and $c \in \{1,\dots,11\}\setminus \bigl(f(v)\cup f(v')\cup f(v'')\bigr)$. This extends the coloring to a $3$-good $11$-coloring of $G$, yielding a contradiction.\end{proof}

\noindent Note that $G$ must contain at least two bounded faces; otherwise, $G$ would be a cycle, which admits a $3$-good $11$-coloring by Lemma \ref{l3}, a contradiction. Let $F$ be a pendant face in $G$ with $C(F) = v_1v_2 \dots v_\ell v_1$, where $\ell \ge 3$. By the definition of a pendant face in a subcubic outerplanar graph, we assume that $v_1$ and possibly $v_\ell$ are the only vertices of degree $3$, with all others having degree $2$.

\begin{claim}
   $\ell \leq 4$.
\end{claim}

\begin{proof}
Suppose, for a contradiction, that $\ell \ge 5$. Consider the graph 
$G' = (G - \{v_3\}) + v_2v_4$. By the minimality of $G$, the graph $G'$ admits a $3$-good $11$-coloring $f$. Without loss of generality, assume that $f(v_2) = 123$ and $f(v_4) = 456$, and let $S = \{7,8,9,10,11\}$. Since $f$ is a $3$-good $11$-coloring of $G'$, the vertex $v_5$ shares exactly one color with $v_2$ and takes its two remaining colors from $S$. Thus, we may assume $f(v_5) = 178$. Similarly, $v_1$ shares exactly one color with $v_4$, say $4$, and takes its two remaining colors from $S$. Hence, the sets $f(v_1), f(v_2), f(v_4)$, and $f(v_5)$ together use at most ten colors, leaving one unused color from $\{1,\dots,11\}$, which we may assume to be $11$. Observe that $v_1$ may share two, one, or no colors with $v_5$. Accordingly, we may assume that 
$f(v_1) = 478$, $479$, or $4910$, respectively. In all cases, we assign $f(v_3) = 8910$.

\noindent It remains to ensure that $v_2$ and $v_4$ share a color in $G$, which we achieve by relabeling $v_4$. Note that $d_G(v_5) \ge 2$.

\noindent Suppose first that $d_G(v_5)=2$, and let $x$ be a neighbor of $v_5$ distinct from $v_4$. In this case, $x \in \{v_1, v_6\}$, and $x$ together with $v_2$ are the only second neighbors of $v_4$ in $G$. Since $f$ is a $3$-good $11$-coloring, $v_4$ and $x$ share exactly one color. Without loss of generality, assume that this color is $4$. If $2 \in f(x)$, we relabel $f(v_4)$ as $256$; otherwise, we relabel it as $245$. In either case, the coloring extends to $G$, yielding a contradiction.

\noindent Thus, $d_G(v_5)=3$. It follows that $\ell=5$, and hence $v_1$ and $v_5$ are adjacent, which fixes $f(v_1)=4910$. Let $x$ be the third neighbor of $v_5$ in $G$, distinct from $v_1$ and $v_4$. In this case, $v_4$ has exactly three second neighbors in $G$, which are $v_1$, $v_2$, and $x$. Observe that $v_1$ and $v_4$ share the color $4$. Moreover, $x$ and $v_4$ share exactly one common color; denote it by $a$.

\noindent Note that $v_1$ and $x$ are not adjacent. Otherwise, $x$ and $v_2$ would be at distance two in $G$, and hence must share exactly one common color other than $1$. Without loss of generality, assume this color is $2$. Then $3 \notin f(x)$. Moreover, $a \neq 4$; assume without loss of generality that $a=5$. Now relabel $v_4$ with $345$, which extends $f$ to $G$, a contradiction.

\noindent Thus, $x$ and $v_1$ are at distance two in $G$, and hence must share exactly one common color. First, suppose that this common color is $4$. If $2 \in f(x)$, we relabel $f(v_4)$ as $2510$; otherwise, we relabel it as $245$. In either case, $f$ extends to $G$, a contradiction.

\noindent Therefore, we may assume, without loss of generality, that $x$ shares the color $5$ with $v_4$ and the color $9$ with $v_1$. Consequently, $|\{2,3\} \cap f(x)| \leq 1$. Suppose, without loss of generality, that $2 \notin f(x)$. We can then relabel $f(v_4)$ as $245$, which once again leads to a contradiction.\end{proof}

\begin{claim}
    $\ell= 3$
\end{claim}

\begin{proof}
Suppose, to the contrary, that $\ell= 4$; then let $G' = (G - \{v_2\}) + v_1v_3$. Without loss of generality, suppose $f(v_1) = 123$ and $f(v_3) = 456$. Set, without loss of generality, $f(v_4) = 789$. Let $y$ be the neighbor of $v_1$ in $G$ other than $v_2$ and $v_4$. Since $f$ is a $3$-good $11$-coloring, $y$ shares exactly one color with $v_3$. Suppose $v_4$ and $y$ are not adjacent; then $y$ also shares exactly one color with $v_4$. Hence, without loss of generality, let $f(y) = 4710$. If $d_G(v_4)=2$, then relabel $v_3$ as $156$ and label $v_2$ as $81011$, a contradiction. Thus, let $y'$ be the third neighbor of $v_4$. Since $G'$ admits a $3$-good $11$-coloring, $y'$ shares one color with $v_1$, say $1$, and one with $v_3$, say $a$. Relabel $f(v_3)$ by replacing one of its colors, distinct from $a$, with $2$, and assign $f(v_2)=81011$, yielding a contradiction. Thus, $v_4$ and $y$ are adjacent. Therefore, without loss of generality, let $f(y) = 41011$. Hence, we can relabel $v_3$ as $426$ and then label $v_2$ as $5710$, a contradiction.\end{proof}

\noindent Then $F$ is a triangle. Suppose first, without loss of generality, that $d_G(v_1)=3$ and $d_G(v_2)=d_G(v_3)=2$. Let $G' = G - \{v_2, v_3\}$. By the minimality of $G$, the graph $G'$ admits a $3$-good $11$-coloring $f$. Let $y$ denote the neighbor of $v_1$ in $G'$. Suppose, without loss of generality, that $f(v_1)=123$ and $f(y)=456$. Assigning the labels $478$ and $5910$ to $v_2$ and $v_3$, respectively, yields a $3$-good $11$-coloring of $G$, a contradiction.

\noindent Therefore, we may assume, without loss of generality, that $d_G(v_1)=d_G(v_2)=3$ and $d_G(v_3)=2$. Let $G' = G - \{v_3\}$. By the minimality of $G$, the graph $G'$ admits a $3$-good $11$-coloring $f$. Suppose, without loss of generality, that $f(v_1)=123$ and $f(v_2)=456$. Let $y_1$ and $y_2$ be the neighbors of $v_1$ and $v_2$, not in $C(F)$, respectively. If $y_1=y_2$, then we can set $f(y_1)=789$. Now, assign the label $71011$ to $v_3$, a contradiction. Thus, $y_1$ and $y_2$ are distinct. Let $S=\{7,8,9,10,11\}$. Because $G'$ admits a $3$-good $11$-coloring, $y_2$ shares exactly one color with $v_1$ and takes its two remaining colors from $S$. Hence, without loss of generality, we may set $f(y_2) = 178$. Similarly, $y_1$ shares exactly one color with $v_2$, say $4$, and takes its two remaining colors from $S$. Therefore, the four sets $f(v_1), f(v_2), f(y_1), f(y_2)$ together involve at most ten colors, leaving one unused in $\{1,\dots,11\}$, which we may assume, without loss of generality, to be $11$. Observe that $y_1$ may share two, one, or no colors with $y_2$. Without loss of generality, we may assume that $f(y_1) = 478$, $479$, or $4910$, respectively. In all cases, we assign $f(v_3) = 8911$ and obtain a $3$-good $11$-coloring of $G$, a contradiction.\end{proof}

\noindent As a consequence of the preceding result, we establish that every subcubic outerplanar graph $G$ satisfies $\tau_3(G) \le 11$. This finding is particularly significant when compared to the broader class of general subcubic graphs, for which the best-known upper bound for the $3$-tone chromatic number is $21$ \cite{Dong2}.

\noindent The bound $\tau_3(G) \le 11$ is sharp. This is evidenced by the fact that $K_4-e$ does not admit a 3-tone 10-coloring. A natural follow-up question is whether excluding $K_4-e$ as a subgraph allows a reduction of the 3-tone chromatic number to $10$. However, this is not the case; the graph illustrated in Figure 1 serves as a counterexample.

\begin{figure}[h]
\centering
\begin{tikzpicture}[scale=1,
every node/.style={circle, fill=black, inner sep=2pt},
label distance=2pt]

\node[label=above left:$v_1$] (v1) at (0,1) {};
\node[label=above:$v_2$]      (v2) at (1,2) {};
\node[label=above:$v_5$]      (v3) at (3,2) {};
\node[label=below right:$v_4$](v4) at (3,0) {};
\node[label=below:$v_3$]      (v5) at (1,0) {};

\draw (v1) -- (v2);
\draw (v2) -- (v3);
\draw (v3) -- (v4);
\draw (v4) -- (v5);
\draw (v5) -- (v1);
\draw (v2) -- (v5);

\end{tikzpicture}
\caption{A subcubic outerplanar graph that is not $3$-tone $10$-colorable.}
\label{f1}
\end{figure}
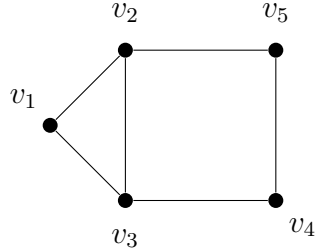
\noindent Since the vertices $v_1, v_2,$ and $v_3$ induce a $K_3$, their color sets must be pairwise disjoint, requiring nine distinct colors. Because $v_4$ is at distance two from $v_1$ and $v_2$, it can receive exactly two colors from the triangle's palette, requiring at least one new color. By symmetry, $v_5$ also requires a new color. Since $v_4$ and $v_5$ are adjacent, these two additional colors must be distinct. Consequently, $\tau_3(G) \ge 9 + 2 = 11$.  

\noindent We conclude this section by proposing a relationship between $\tau_t(G)$ and $\tau_{t-1}(G)$ for subcubic outerplanar graphs.

\begin{conjecture}
For every subcubic outerplanar graph $G$ and every integer $t\ge 3$, \[\tau_t(G)\le \tau_{t-1}(G)+t+1.\]
\end{conjecture}

\section{$2$-tone coloring of Halin graphs}

This section studies the $2$-tone chromatic number of Halin graphs. We first show that every cubic Halin graph of order $n \ge 6$ is $2$-tone $7$-colorable and then derive an upper bound on the $2$-tone chromatic number for Halin graphs with arbitrary maximum degree. We remark that, by definition, every Halin graph $H$ has maximum degree $\Delta(H)\ge 3$.

\subsection{Cubic Halin Graphs}

\noindent We begin with the cubic case by recalling the following definitions and a known result on trees, which will be used in the proof of our main theorem.

\begin{definition}[\cite{C1}]
A {\em partial $t$-tone $k$-coloring} of a graph $G$ is a function $f: S \rightarrow {[k] \choose t}$, with $S \subseteq V(G)$, such that $|f(u) \cap f(v)| < d_G(u,v)$ whenever $u,v \in S$. 
\end{definition}

\begin{definition}[\cite{C1}]
Let $f$ be a partial $2$-tone coloring of a graph $G$, and let $v$ be an uncolored vertex of $G$. A color is said to be \emph{free at $v$} if it does not appear in the label of any neighbor of $v$. The labels obtained using only free colors at $v$ are called \emph{candidate} labels for $v$.
\end{definition}

\begin{lemma}[\cite{BP,F}]\label{l4}
For any nontrivial tree $T$ with a maximum degree $\Delta$, 
\[
\tau_2(T) = \left\lceil \frac{5 + \sqrt{8\Delta + 1}}{2} \right\rceil.
\]
\end{lemma}

\begin{theorem}\label{t1'}
Every cubic Halin graph of order $n \ge 6$ is $2$-tone $7$-colorable.
\end{theorem}

\begin{proof}

Let $H = T \cup C$ be a cubic Halin graph with $n \ge 6$ vertices, where $C = x_1 x_2 \cdots x_\ell x_1$. Define $T' = T -\{x_1, \dots, x_\ell\}$. Then $T'$ is a tree with maximum degree at most $3$. By Lemma~\ref{l4}, $T'$ admits a $2$-tone $5$-coloring $f$. For any two distinct vertices $u,v \in V(T')$ with $d_H(u,v)\le 2$, we have $d_{T'}(u,v)\le d_H(u,v)$. Hence, $f$ is a partial $2$-tone $5$-coloring of $H$.

\noindent We observe that the tree $T$ must contain at least two internal vertices. If this were not the case, $H$ would be isomorphic to $K_4$, contradicting the assumption $n \ge 6$. Root the tree $T$ at an arbitrary internal vertex, and let $x$ be an internal vertex at maximum distance from the root in $T$. This selection, combined with the property that $H$ is a cubic Halin graph, implies that $x$ has exactly one internal neighbor (its parent, say $y$) and adjacent to exactly two leaves of $T$ that necessarily appear consecutively on $C$. Without loss of generality, let these two leaves be $x_1$ and $x_\ell$. For $i \in \{2, \dots, \ell-1\}$, we denote by $y_i$ the parent of $x_i$ in $H$.

\noindent First, we extend the coloring $f$ to $x_1$. This vertex has exactly one colored neighbor, namely its parent $x$, and $f(x) \subseteq \{1,\dots,5\}$.
Consequently, at least three colors from $\{1,\dots,5\}$ are free at $x_1$, yielding at least three candidate labels for $x_1$. Moreover, among the second neighbors of $x_1$ in $H$, at most two are internal vertices, namely $y$ and $y_2$. It follows that $x_1$ can be assigned one of its candidate labels.

\noindent Let us define an \emph{$L$-label} as a set consisting of exactly one color from $\{6,7\}$ and exactly one color from $\{1,\dots,5\}$. Accordingly, a leaf assigned an $L$-label is referred to as an \emph{$L$-leaf}. We color the vertices $x_2,\dots,x_{\ell-1}$ iteratively in cyclic order, assigning each vertex an $L$-label so that, after each step, the coloring remains valid. In particular, upon completion of this process, the coloring restricted to $V(H)\setminus\{x_\ell\}$ is a partial $2$-tone $7$-coloring of $H$.

\noindent We now consider $x_2$. At this stage, $x_3$ is uncolored; hence, both colors $6$ and $7$ are free at $x_2$. Since $f(y_2) \cup f(x_1) \subseteq \{1,\dots,5\}$, there exists at least one color in $\{1,\dots,5\}$ that is free at $x_2$. Therefore, $x_2$ admits at least two candidate $L$-labels. As no $L$-label has been assigned yet, $x_2$ has no second neighbor that is an $L$-leaf. Thus, $x_2$ can be assigned one of its candidate $L$-labels without conflict.

\noindent Now let $i \in \{3,\dots,\ell-1\}$ and suppose that the vertices
$x_2,\dots,x_{i-1}$ have already been assigned $L$-labels. At this stage, the vertex $x_{i+1}$ is uncolored. Since $f(y_i) \subseteq \{1,\dots,5\}$ and $x_{i-1}$ is an $L$-leaf, there are at least two free colors from $\{1,\dots,5\}$ and exactly one free color from $\{6,7\}$ at $x_i$, yielding at least two candidate $L$-labels for $x_i$. We next examine the second neighbors of $x_i$ that are leaves. Each of $x_{i-1}$ and $x_{i+1}$ is adjacent to exactly one leaf other than $x_1$. Moreover, since $H$ is a Halin graph, any leaf adjacent to $y_i$ must be either $x_{i-1}$ or $x_{i+1}$. Hence, $x_i$ has at most two leaves as second neighbors, namely $x_{i-2}$ and $x_{i+2}$. Among these, at most one can be an $L$-leaf: if such a vertex exists, it must be $x_{i-2}$, since $x_{i+2}$ is either uncolored at this stage or coincides with $x_1$. Therefore, $x_i$ has at most one second neighbor that is an $L$-leaf, and consequently, $x_i$ can be assigned one of its candidate $L$-labels without conflict.

\noindent It remains to assign a valid label to the final leaf, $x_\ell$. Let $L$ denote the $L$-label previously assigned to $x_{\ell-1}$. Without loss of generality, we may assume the following assignments: $f(x) = 12$, $f(x_1) = 34$, and $f(y) = 35$. Additionally, we assume $6 \in L$. We consider two cases according to whether $L$ uses a color from the set $\{1,2\}$ or not.

\noindent Case 1. $1,2\notin L$. Equivalently $L$ is one of $\{36,46,56\}$.\\
\noindent Relabel $x$ by $67$; this is possible as no vertex is assigned the label 67. If $f(y_{\ell-1})\neq 12$, then label $x_\ell$ by $12$, and we are done. So suppose $f(y_{\ell-1})=12$. Since $x_1$ is adjacent to the $L$-leaf $x_2$ and has at most one internal second neighbor, namely $y_2$, whose label may contain $1$ or $2$, at least one label from $\{13,14,23,24\}$ can be assigned to $x_1$. Without loss of generality, assign $23$ to $x_1$. Now, if $L\in\{36,56\}$, label $x_\ell$ by $14$, and if $L=46$, label $x_\ell$ by $15$.

\noindent Case 2. $L$ contains one color from $\{1,2\}$. Assume, without loss of generality, $1\in L$; that is, $L=16$.\\
\noindent Relabel $x$ by $67$ as in Case 1. If $f(y_{\ell-1})\neq 25$, then label $x_\ell$ by $25$ and finish. So suppose $f(y_{\ell-1})=25$. Consider now the $L$-label of $x_1$. If the $L$-label of $x_2$ does not contain $2$ and $f(y_2)\neq 25$, then we may relabel $x_1$ by $25$, and then label $x_\ell$ by $34$. Hence, assume instead that either the $L$-label of $x_2$ contains $2$ or $f(y_2)=25$. Suppose first that the $L$-label of $x_2$ contains $2$. Since $x_1$ has at most one internal second neighbor, $y_2$, whose label may contain $1$ or $2$, one label from $\{13,14\}$ can be assigned to $x_1$. Without loss of generality, assign $13$ to $x_1$; then label $x_\ell$ by $24$. Thus, $f(y_2)=25$, relabel $x_1$ by $45$, and label $x_\ell$ by $23$. \end{proof}

\noindent This result confirms part (2) of Conjecture~\ref{conj} for cubic Halin graphs of order $n \ge 6$. However, no cubic Halin graph requiring seven colors is currently known, suggesting that the above bound may not be tight. This observation motivates the following conjecture.

\begin{conjecture}\label{c}
Every cubic Halin graph of order $n\ge 6$ is $2$-tone $6$-colorable.
\end{conjecture}

\noindent Since cubic Halin graphs are $(K_4-e)$-free, the above conjecture, if confirmed, would confirm part (3) of Conjecture~\ref{conj} for this class of graphs.

\subsection{Halin Graphs}

\noindent For planar graphs, Cranston and LaFayette~\cite{C2} proved that for any planar graph $G$, the $2$-tone chromatic number satisfies
\[
\tau_2(G) \le 
\left\lfloor \sqrt{4\Delta(G)+50.25}+31.1 \right\rfloor
\le
\left\lfloor \sqrt{4\Delta(G)}+36.5 \right\rfloor .
\]
Moreover, for sufficiently large $\Delta(G)$, they established the stronger bound
\[
\tau_2(G)\le 
\max\left\{41,\left\lfloor \sqrt{4\Delta(G)+50.25}+11.5 \right\rfloor \right\}.
\]
They further conjectured~\cite{C2} that there exists a constant $C$ such that $\tau_2(G) \le \sqrt{3\Delta(G)} + C$ for every planar graph $G$. 

\noindent We start with a lemma for the wheel graph $W_d$, which is a basic special case of Halin graphs. 

\begin{lemma}(\cite{B1})\label{l5}
For the wheel graph $W_d = C_d + K_1$, the 2-tone chromatic number is:
\[
\tau_2(W_d)=
\begin{cases}
7, & \text{if } d \in \{5, 6, 8, 9\},\\[2mm]
8, & \text{if } d \in \{3, 4, 7, 10, \dots, 15\},\\[2mm]
\left\lceil \dfrac{5+\sqrt{1+8d}}{2} \right\rceil, & \text{if } d \ge 11.
\end{cases}
\]
\end{lemma}

\begin{theorem}\label{t2'}
Let $H = T \cup C$ be a Halin graph. Then
\[
\tau_2(H)\le 
\max\left\{
10,\;
\left\lceil \frac{13+\sqrt{8\Delta(H)-15}}{2} \right\rceil
\right\}.
\]
\end{theorem}

\begin{proof}
Suppose, for the sake of contradiction, that the theorem is false. 
Let $H = T \cup C$ be a counterexample with the minimum number of vertices, 
where $C = x_1 x_2 \cdots x_\ell x_1$ is a cycle. Set $k = \max\left\{
10,\;
\left\lceil \frac{13+\sqrt{8\Delta(H)-15}}{2} \right\rceil
\right\}$.

\noindent By Lemma~\ref{l5}, $H$ is not a wheel graph; hence, the tree $T$ contains at least two internal vertices. Root $T$ at an arbitrary internal vertex, and let $x$ be an internal vertex at maximum distance from the root in $T$. By this choice, $x$ has exactly one internal neighbor (its parent, say $p$) and all its remaining neighbors are leaves of $T$ that appear consecutively on $C$. Denote, without loss of generality, these leaves by $x_1,x_2,\dots,x_m$. Since $T$ has no vertices of degree two, we have $m\ge 2$.

\begin{claim} 
 $m\ge 3$.
\end{claim}

\begin{proof}
Suppose, to the contrary, that $m= 2$. Define a smaller graph $H' = (H - \{x_1,x_2\})+ \{x_\ell x,\, x x_3\}$. $H'$ is again a Halin graph: the vertex $x$ becomes a leaf on the new cycle $C'$, and its degree in the new tree $T' = T - \{x_1,x_2\}$ remains three. By the minimality of $H$, $H'$ admits a valid $2$-tone $k$-coloring $f$.

\noindent We now extend $f$ to $H$. First, consider $x_1$. Its neighbors in $H$ are $\{x,x_\ell,x_2\}$; at this stage, $x_2$ is uncolored. Thus, the number of available labels for $x_1$ satisfying the adjacency condition is at least $\binom{k-4}{2}$. Moreover, $x_1$ has at most four vertices at distance two, and since $k\ge 10$, we have $\binom{k-4}{2}>4$. Hence, at least one label exists for $x_1$.

\noindent Next, consider $x_2$. Its neighbors in $H$ are $\{x,x_1,x_3\}$, all of which are now colored and together use at least six colors. Thus, the number of labels available for $x_2$ satisfying the adjacency condition is at least $\binom{k-6}{2}$. The vertex $x_2$ has at most four second neighbors, and since $k\ge 10$, we have $\binom{k-6}{2}>4$. Therefore, a label can be assigned to $x_2$.

\noindent This yields a valid $2$-tone $k$-coloring of $H$, contradicting the assumption that $H$ is a minimal counterexample.\end{proof}

\noindent Thus, the neighbors of $x$ in $T$ are $\{p,x_1,x_2,\dots,x_m\}$ with $m\ge 3$. In this case, $d_H(x)=m+1\ge 4$. Construct a smaller graph $H' = (H - \{x_2\}) + \{x_1x_3\}$. Since $m\ge 3$,  $x$ remains adjacent to at least two leaves ($x_1$ and $x_3$) and to its parent $p$ in $H'$, and hence $d_{H'}(x)\ge 3$. Thus, $H'$ is a Halin graph with $|V(H')|< |V(H)|$. By the minimality of $H$, $H'$ admits a valid $2$-tone $k'$-coloring,
where
\[
k'=\max\left\{
10,\;
\left\lceil \frac{13+\sqrt{8\Delta(H')-15}}{2} \right\rceil
\right\}.
\]
Since $\Delta(H')\le \Delta(H)$, we have $k'\le k$. Hence, $H'$ admits a $2$-tone $k$-coloring $f$.

\noindent We now extend $f$ to $x_2$ in $H$. The vertices at distance two from $x_2$ are:
\begin{itemize}
    \item from $x$: the parent $p$ and the leaves $\{x_4,\dots,x_m\}$;
    \item from $x_1$: the neighbor of $x_1$ on $C$ other than $x_2$, namely $x_\ell$;
    \item from $x_3$: the neighbor of $x_3$ on $C$ other than $x_2$, namely $x_4$.
\end{itemize}
Thus, $|N_H^2(x_2)| = |\{p,x_4,x_5,\dots,x_m\}\cup \{x_\ell\}|\le (d_H(x)-3)+1\le \Delta(H)-2$. Since $x_2$ has three colored neighbors, at most $6$ colors are forbidden by adjacency, leaving at least $\binom{k-6}{2}$ possible labels for $x_2$. By the definition of $k$ and since $\Delta(H)\ge 3$, we have $\binom{k-6}{2}>\Delta(H)-2$, ensuring the existence of a label for $x_2$. This again yields a $2$-tone $k$-coloring of $H$, a contradiction.\end{proof}

\noindent While Theorem \ref{t2'} provides a general upper bound on $\tau_2(H)$ in terms of $\Delta(H)$, Theorem \ref{t1'} yields a strictly stronger result in the cubic case.

\noindent Asymptotically, the upper bound for general planar graphs grows as $\sqrt{4\Delta}$. In contrast, our bound for Halin graphs scales as $\frac{\sqrt{8\Delta}}{2}=\sqrt{2\Delta}$. Moreover, our bound is strictly stronger than the conjectured bound $\sqrt{3\Delta}$ proposed by Cranston and LaFayette.


\end{document}